\documentclass{article}

\usepackage[all,ps,arc]{xy}
\usepackage{amsmath,amssymb,latexsym}
\usepackage{amsfonts}
\usepackage{amsthm}

\newtheorem{Theorem}{Theorem}[section]
\newtheorem{Lemma}[Theorem]{Lemma}
\newtheorem{Proposition}[Theorem]{Proposition}

\newtheorem{Corollary}[Theorem]{Corollary}

\newcommand{\bC}{\ensuremath{\mathbb C}}
\newcommand{\bD}{\ensuremath{\mathbb D}}
\newcommand{\bE}{\ensuremath{\mathbb E}}
\newcommand{\bG}{\ensuremath{\mathbb G}}
\newcommand{\bH}{\ensuremath{\mathbb H}}
\newcommand{\bR}{\ensuremath{\mathbb R}}
\newcommand{\bS}{\ensuremath{\mathbb S}}

\newcommand{\cC}{\ensuremath{\mathcal C}}

\DeclareMathOperator{\Dec}{Dec}
\DeclareMathOperator{\eps}{\epsilon}

\newcommand\Eq{\ensuremath{\mathsf{Eq}}}
\newcommand\Gpd{\ensuremath{\mathsf{Gpd}}}
\newcommand\Cat{\ensuremath{\mathsf{Cat}}}

\newdir{|>}{{}*!/8pt/@{|}*!/3.5pt/:(1,-.2)@^{>}*!/3.5pt/:(1,+.2)@_{>}}
\newdir{ >}{{}*!/-10pt/@{>}}
\newdir{ |>}{!/10pt/{}*!/4.5pt/@{|}*:(1,-.2)@^{>}*:(1,+.2)@_{>}}

\begin{document}

\newenvironment{changemargin}[2]{\begin{list}{}{
\setlength{\topsep}{0pt}
\setlength{\leftmargin}{0pt}
\setlength{\rightmargin}{0pt}
\setlength{\listparindent}{\parindent}
\setlength{\itemindent}{\parindent}
\setlength{\parsep}{0pt plus 1pt}
\addtolength{\leftmargin}{#1}\addtolength{\rightmargin}{#2}
}\item}{\end{list}}

\title{A characterization of final functors between internal groupoids in exact categories}
\author{A.\ S.\ Cigoli}

\maketitle

\begin{abstract}
This paper provides several characterizations of final functors between internal groupoids in a Barr-exact category \cC. In particular, it is proved that an internal functor between groupoids is final if and only if it is full and essentially surjective.
\end{abstract}


\section{Introduction}

The \emph{comprehensive factorization} of a functor was introduced by Street and Walters in \cite{StW}, where they showed that any functor $F \colon \bC \to \bD$ between arbitrary categories is the composite of an initial functor followed by a discrete opfibration, and these two classes of morphisms form a factorization system for \Cat. By duality, a similar factorization can be obtained by means of a final functor followed by a discrete fibration, and the latter too is known under the name of comprehensive factorization. 

Let us recall that a functor $F \colon \bC \to \bD$ is called \emph{final} if, for any object $y$ in \bD, the comma category $(y \downarrow F)$ is non-empty and connected. On the other hand, a \emph{discrete fibration} is a functor $F \colon \bC \to \bD$ such that, for any object $x$ in \bC\ and any arrow $g \colon y \to F(x)$ in \bD, there exists a unique arrow $f$ in \bC\ with codomain $x$ and such that $F(f)=g$. Initial functors and discrete opfibrations are the obvious dualizations of the previous concepts.

It is an easy observation that if we restrict our attention to groupoids, then initial and final functors coincide, and the same is true for (discrete) fibrations and opfibrations. Moreover the following elementary result holds.

\begin{Proposition} \label{prop:final=fes}
A functor between groupoids is final if and only if it is full and essentially surjective.
\end{Proposition}

\begin{proof}
Let $F \colon \bC \to \bD$ be a functor between groupoids. Asking that, for any object $y$ in \bD, the comma category $(y \downarrow F)$ is non-empty is equivalent to asking that $F$ is essentially surjective.

Suppose now that $F$ is full and consider two objects $(g\colon y \to F(x),x)$ and $(g'\colon y \to F(x'),x')$ in the comma category $(y \downarrow F)$. Then the arrow $g'g^{-1}\colon F(x) \to F(x')$ has an inverse image $f\colon x \to x'$ in \bC, and $(g'g^{-1},f)$ is an arrow in $(y \downarrow F)$ connecting $(g,x)$ and $(g',x')$. Conversely, if $F$ is final, then given an arrow $g\colon F(x) \to F(x')$, an inverse image of $g$ is the second projection of a connecting isomorphism between $(1_{F(x)},x)$ and $(g,x')$ in the comma category $(F(x) \downarrow F)$
\end{proof}

An internal version of the comprehensive factorization system for functors between groupoids in a Barr-exact category has been developed by Bourn in \cite{Bourn87}, where he provided an explicit construction of the above factorization by means of the \emph{d\'ecalage} functor. In that paper, Bourn takes discrete fibrations as basic notion and defines final functors as their orthogonal class. As far as we know, the characterization of Proposition \ref{prop:final=fes} has no internal counterpart in the literature. It is the aim of the present paper to fill this gap.

Throughout the paper \cC\ will be a Barr-exact category.


\section{Internal groupoids}

We fix here some notation and recall some basic facts about internal groupoids in Barr-exact categories.

An internal category in \cC\ is given by a diagram
\[
\xymatrix{
H_1 \times_{(c,d)} H_1 \ar@<1ex>[r]^-{p_2} \ar@<-1ex>[r]_-{p_1} \ar[r]|-{m_H} & H_1 \ar@<1ex>[r]^-{c_H} \ar@<-1ex>[r]_-{d_H} & H_0 \ar[l]|-{e_H}
}
\]
satisfying the usual axioms for a category. It is a groupoid when, in addition, it comes equipped with an inversion morphism for arrows
\[
i_H\colon H_1 \to H_1
\]
satisfying the well known axioms. We will avoid the use of subscripts as far as no confusion arises.

An internal functor, denoted by $F\colon \bH \to \bG$, is given by a pair of morphisms $(f_0,f_1)$ in \cC\ such that $f_0d=df_1$, $f_0c=cf_1$, $ef_0=f_1e$ and $m_G(f_1\times_{f_0}f_1)=f_1m_H$:
\[
\xymatrix{
H_1 \ar@<1ex>[d]^c \ar@<-1ex>[d]_d \ar[r]^{f_1} & G_1 \ar@<1ex>[d]^c \ar@<-1ex>[d]_d \\
H_0 \ar[u]|e \ar[r]_{f_0} & G_0 \ar[u]|e
}
\]
We denote by $\Gpd(\cC)$ the category of internal groupoids in \cC\ and functors between them.

\subsection{Some relevant classes of functors} \label{sec:functclasses}

An internal \emph{discrete fibration} is a functor $F\colon \bH \to \bG$ of internal categories such that the following square of solid arrows is a pullback:
\[
\xymatrix{
H_1 \ar@<.7ex>[d]^c \ar@<-.7ex>@{-->}[d]_d \ar[r]^{f_1} & G_1 \ar@<.7ex>[d]^c \ar@<-.7ex>@{-->}[d]_d \\
H_0 \ar[r]_{f_0} & G_1
}
\]
In the case of groupoids, thanks to the inversion morphisms, the commutative square with dashed downwarded arrows is a pullback too, hence $F$ is also a discrete opfibration. It is easy to prove that discrete (op)fibrations are pullback stable.

Given a groupoid \bG\ and a morphism $f\colon X \to G_0$ in \cC, the following procedure yields an internal functor. Consider the pullback
\[
\xymatrix{
P \ar[d]_{p_1} \ar[r]^{p_2} & G_1 \ar[d]^{(d,c)} \\
X \times X \ar[r]_-{f \times f} & G_0 \times G_0
}
\]
Then the following is an internal functor between groupoids
\[
\xymatrix{
P \ar@<.7ex>[d]^{c} \ar@<-.7ex>[d]_{d} \ar[r]^{p_2} & G_1 \ar@<.7ex>[d]^c \ar@<-.7ex>[d]_d \\
X \ar[r]_{f} & G_0
}
\]
(where the $d$ and $c$ on the left are the composites of $p_1$ with the product projections) which is indeed a cartesian lifting of $f$ at \bG, with respect to the fibration $()_0\colon \Gpd(\cC)\to \cC$ associating with any groupoid in \cC\ its object of objects. So we are allowed to denote by $f^*\bG$ the domain of the functor $(f,p_2)$. The square above is also called the joint pullback of $(d,c)$ along $f$.

In particular, given a functor $F\colon \bH\to\bG$ in $\Gpd(\cC)$, one can factor $F$ through $f_0^*\bG$, as in the following diagram:
\[
\xymatrix{
H_1 \ar@/^3ex/[rr]^{f_1} \ar[r]_{\phi_F} \ar@<.7ex>[d]^c \ar@<-.7ex>[d]_d & P \ar@<.7ex>[d]^{c} \ar@<-.7ex>[d]_{d} \ar[r]_{p_2} & G_1 \ar@<.7ex>[d]^c \ar@<-.7ex>[d]_d \\
H_0 \ar@{=}[r] & H_0 \ar[r]_{f_0} & G_0
}
\]
$F$ is said to be \emph{full} if $\phi_F$ is a regular epimorphism and \emph{faithful} if $\phi_F$ is a monomorphism. In fact, the above procedure yields a factorization system for $\Gpd(\cC)$ given by bijective on objects and fully faithful functors.

\begin{Proposition} \label{prop:pbff}
Let the diagram below be a pullback in $\Gpd(\cC)$:
\[
\xymatrix{
\bH \times_{\bG} \bH' \ar[r]^-{\overline{F}} \ar[d]_{\overline{F}'} & \bH' \ar[d]^{F'} \\
\bH \ar[r]_{F} & \bG
}
\]
The following properties hold:
\begin{enumerate}
 \item If $F$ is full, $\overline{F}$ is full, and the converse is true if the arrow component $f_1'$ of $F'$ is a regular epimorphism;
 \item If $F$ is faithful, $\overline{F}$ is also faithful.
\end{enumerate}
\end{Proposition}

\begin{proof}
Let us consider the following commutative diagram, where the front and back faces of the cube are the pullbacks yielding the fully faithful liftings of $f_0$ and $\overline{f}_0$ at \bG\ and $\bH'$ respectively:
\[
\xymatrix@=4ex{
P' \ar[rr]^{u'} \ar[dd] \ar[dr]_{v} & & H_1' \ar[dd]_(.3){(d,c)} \ar[dr]^{f_1'} \\
& P \ar[rr]_(.3){u} \ar[dd] & & G_1 \ar[dd]^{(d,c)} \\
(H_0 \times_{G_0} H_0') \times (H_0 \times_{G_0} H_0') \ar[dr]_{\overline{f}_0'\times \overline{f}_0'} \ar[rr]_(.75){\overline{f}_0\times \overline{f}_0} & & H_0' \times H_0' \ar[dr]^{f_0'\times f_0'} \\
& H_0 \times H_0 \ar[rr]_{f_0 \times f_0} & & G_0 \times G_0
}
\]
Since the bottom face is a pullback, then so is the top face. As a consequence, in the diagram below, since the whole rectangle and the right square are pullbacks, so is the square on the left hand side.
\[
\xymatrix{
H_1 \times_{G_1} H_1' \ar[d]_{\overline{f}_1'} \ar[r]^-{\phi_{\overline{F}}} & P' \ar[d]_v \ar[r]^{u'} & H_1' \ar[d]^{f_1'} \\
H_1 \ar[r]_{\phi_{F}} & P \ar[r]_{u} & G_1
}
\]
Now the theses follow by definition of full and faithful functor and from the fact that monomorphisms and regular epimorphisms are pullback stable in \cC.
\end{proof}

Finally, given a functor $F\colon\bH\to\bG$, consider the following pullback in \cC:
\[
\xymatrix{
E_0 \ar[r]^-{\overline{f_0}} \ar[d]_{p_1} & G_1 \ar[d]^{d} \\
H_0 \ar[r]_{f_0} & G_0
}
\]
$F$ is said to be \emph{essentially surjective} if the composite $c\cdot\overline{f_0}\colon E_0\to G_0$ is a regular epimorphism.

\subsection{Support and connected components}

For any groupoid \bH, the pair $(d,c)$ factors through an equivalence relation $\Sigma\bH$ on $H_0$
\[
\xymatrix{
\Sigma H_1 \ar@<1ex>[r]^{r_2} \ar@<-1ex>[r]_{r_1} & H_0 \ar[l]|{s}
} \,,
\]
where $\Sigma H_1$ is the regular image of $H_1$ in $H_0 \times H_0$
\[
\xymatrix{
H_1 \ar@{->>}[r]^-{\sigma_H} \ar@/_3ex/[rr]_{(d,c)} & \Sigma H_1 \ar@{ >->}[r]^-{(r_1,r_2)} & H_0 \times H_0
} \,.
\]
\cC\ being Barr-exact, $\Sigma\bH$ is effective and we will denote by $q_H \colon H_0 \to \pi_0(\bH)$ its quotient, which is also the coequalizer of $(d,c)$.

In fact, the above procedure defines three functors:
\[
\Sigma \colon \Gpd(\cC) \to \Eq(\cC) \,, \qquad \bH \mapsto \Sigma\bH
\]
the \emph{support} functor, sending each groupoid \bH\ to its associated equivalence relation, also called the support of \bH.
\[
Q \colon \Eq(\cC) \to \cC 
\]
sending each equivalence relation to its quotient object.
\[
\pi_0 \colon \Gpd(\cC) \to \cC
\]
the \emph{connected components} functor, which is nothing but the composite $\pi_0=Q\cdot\Sigma$, sending each groupoid \bH\ to its object of connected components $\pi_0(\bH)$. From now on, we will use the notation $\pi_0$ also for the quotient of an equivalence relation, identifying it with the associated internal groupoid.
\medskip

The next two results are based on Proposition 1.1 in \cite{Bourn03} and will be useful afterwards.

\begin{Proposition} \label{prop:ff<->mono}
A functor $F\colon\bR\to\bS$ between internal equivalence relations in \cC\ is fully faithful if and only if $\pi_0(F)$ is monomorphic.
\end{Proposition}

\begin{proof}
Let us draw the components of $F$ vertically, and compute $\pi_0(F)$ as the induced arrow between the quotient objects of the domain and codomain:
\[
\xymatrix{
R_1 \ar@<.7ex>[r]^{r_2} \ar@<-.7ex>[r]_{r_1} \ar[d]_{f_1} & R_0 \ar[d]^{f_0} \ar@{->>}[r] & \pi_0(\bR) \ar[d]^{\pi_0(F)} \\
S_1 \ar@<.7ex>[r]^{r_2} \ar@<-.7ex>[r]_{r_1} & S_0 \ar@{->>}[r] & \pi_0(\bS)
}
\]
Since the two rows in the above diagram are exact forks (a regular epimorphism with its kernel pair), the thesis follows by Proposition 1.1 in \cite{Bourn03}.
\end{proof}

\begin{Corollary} \label{cor:full->mono}
If a functor $F\colon\bH\to\bG$ between internal groupoids in \cC\ is full, then $\pi_0(F)$ is monomorphic. The converse is true if\, \bG\ is an internal equivalence relation.
\end{Corollary}

\begin{proof}
Let us consider the following commutative diagram, where the left hand side squares are pullbacks (i.e.\ $P_1$ and $P_2$ yield the full and faithful liftings of $f_0$ at $\Sigma G_1$ and $G_1$ respectively):
\[
\xymatrix{
H_1 \ar@{->>}[d]_{\sigma_H} \ar@/^3ex/[rr]^{f_1} \ar[r]_{\phi_F} & P_2 \ar[r] \ar@{->>}[d] & G_1 \ar@{->>}[d]^{\sigma_G} \\
\Sigma H_1 \ar@{ >->}[d] \ar@/^4ex/[rr]^(.7){\Sigma f_1} \ar[r]_{\phi_{\Sigma F}} & P_1 \ar[r] \ar@{ >->}[d] & \Sigma G_1 \ar@{ >->}[d] \\
H_0 \times H_0 \ar@{=}[r] & H_0 \times H_0 \ar[r]_{f_0 \times f_0} & G_0 \times G_0
}
\]
If $F$ is full, $\phi_F$ is a regular epimorphism by definition, hence $\phi_{\Sigma F}$ is a regular epimorphism and a monomorphism at the same time, so it is an isomorphism, i.e.\ $\Sigma F$ is full and faithful. Hence $\pi_0(F)$ is monomorphic by Proposition \ref{prop:ff<->mono}.

Conversely, if $\pi_0(F)$ is a monomorphism then $\Sigma F$ is fully faithful, i.e.\ $\phi_{\Sigma F}$ is an isomorphism. If in addition \bG\ is an equivalence relation, then $\sigma_G$ is an isomorphism, hence $P_2\cong P_1$ and $\phi_F$ is a regular epimorphism, being isomorphic to $\sigma_H$. By definition $F$ is then full.
\end{proof}

\begin{Proposition} \label{prop:es}
A functor $F\colon\bH\to\bG$ between internal groupoids in \cC\ is essentially surjective if and only if $\pi_0(F)$ is a regular epimorphism.
\end{Proposition}

\begin{proof}
It suffices to focus on the following commutative diagram:
\[
\xymatrix@=7ex{
H_0 \times_{(f_0,d)} G_1 \ar@{}[dr]|{(a)} \ar@/_10ex/[dd]_{p_1} \ar[r]^-{\overline{f_0}} \ar[d]^{1\times\sigma_G} & G_1 \ar@/_5ex/[dd]^(.7){d} \ar[r]^{c} \ar[d]^{\sigma_G} & G_0 \ar@{=}[d] \\
H_0 \times_{(f_0,r_1)} \Sigma G_1 \ar@{}[dr]|{(b)} \ar[d] \ar[r]^(.6){p_2} & \Sigma G_1 \ar@{}[dr]|{(c)} \ar[d]^{r_1} \ar[r]_{r_2} & G_0 \ar[d]^{q_G} \\
H_0 \ar[r]_{f_0} & G_0 \ar[r]_{q_G} & \pi_0(\bG)
}
\]
If $F$ is essentially surjective, then $c\overline{f_0}$ is a regular epimorphism, and so is $q_Gc\overline{f_0}=q_Gf_0p_1=\pi_0(F)q_Hp_1$, hence $\pi_0(F)$ is also a regular epimorphism.

Conversely, if $\pi_0(F)$ is a regular epimorphism, then so is $q_Gf_0=\pi_0(F)q_H$. Now let us observe that $(c)$, $(b)$, $(a)+(b)$, and hence $(a)$ are all pullbacks. So $r_2p_2$ is a regular epimorphism as a pullback of $q_Gf_0$, and $1\times \sigma_G$ is a regular epimorphism as a pullback of $\sigma_G$. Then their composite $c\overline{f_0}$ is a regular epimorphism and $F$ is essentially surjective.
\end{proof}


\section{The comprehensive factorization}

We borrow from \cite{Bourn87} the definition and some needed results about the shift functor $\Dec\colon \Gpd(\cC)\to\Gpd(\cC)$.

Let us recall that, for any groupoid \bH, the pullback $H_1 \times_{(c,d)} H_1$ is also isomorphic to the kernel pair $R_d$ of $d$ (or the kernel pair $R_c$ of $c$):
\[
\xymatrix{
R_d \ar[r]_-\sim^-{(ir_1,r_2)} & H_1 \times_{(c,d)} H_1 & R_c \ar[l]^-\sim_-{(r_1,ir_2)}
}
\]
For the sake of convenience, we denote
\[
\overline{d}=m(r_1,ir_2)\colon R_c \to H_1 \qquad \mbox{and} \qquad \overline{c}=m(ir_1,r_2)\colon R_d \to H_1
\] 
We define here $\Dec$ by means of $R_c$ as the functor associating with any groupoid \bH\ in \cC\ the following internal groupoid:
\[
\xymatrix{
R_c \ar@<1ex>[r]^-{r_2} \ar@<-1ex>[r]_-{r_1} & H_1 \ar[l]|-{s}
} \,,
\]
and we denote by $\eps\bH\colon\Dec\bH\to\bH$ the internal discrete fibration
\[
\xymatrix@C=7ex{
R_c \ar@<1ex>[d]^{r_2} \ar@<-1ex>[d]_{r_1} \ar[r]^{\overline{d}} & H_1 \ar@<1ex>[d]^c \ar@<-1ex>[d]_d \\
H_1 \ar[u]|s \ar[r]_{d} & H_0 \ar[u]|e
}
\]
The following is an exact fork in $\Gpd(\cC)$:
\[
\xymatrix{
\Dec^2\bH \ar@<.7ex>[r]^{\eps{\Dec\bH}} \ar@<-.7ex>[r]_{\Dec\eps\bH} & \Dec\bH \ar[r]^-{\eps\bH} & \bH
} \,.
\]

Following \cite{BR}, we give here a description of the comprehensive factorization in $\Gpd(\cC)$. Further details can be found also in \cite{Bourn87}. Let $F\colon\bH\to\bG$ be an internal functor, then the pair $(\Dec F,\Dec^2 F)$ gives rise to a morphism between equivalece relations in $\Gpd(\cC)$:
\[
\xymatrix@C=9ex{
\Dec^2\bH \ar@<.7ex>[d]^{\eps{\Dec\bH}} \ar@<-.7ex>[d]_{\Dec\eps\bH} \ar[r]^{\Dec^2 F} & \Dec^2\bG \ar@<.7ex>[d]^{\eps{\Dec\bG}} \ar@<-.7ex>[d]_{\Dec\eps\bG} \\
\Dec\bH \ar[r]_{\Dec F} & \Dec\bG
}
\]
We consider the following factorization of the above functor, where all the right hand side squares are pullbacks:
\begin{equation} \label{diag:Dec2}
\begin{aligned}
\xymatrix{
\Dec^2\bH \ar@<.7ex>[d]^{\eps{\Dec\bH}} \ar@<-.7ex>[d]_{\Dec\eps\bH} \ar[r] & \bR_{\Delta} \ar@<.7ex>[d] \ar@<-.7ex>[d] \ar[r] & \Dec^2\bG \ar@<.7ex>[d]^{\eps{\Dec\bG}} \ar@<-.7ex>[d]_{\Dec\eps\bG} \\
\Dec\bH \ar[d]_{\eps\bH} \ar[r] & \bE \ar[r]^-{\overline{F}} \ar[d]_{\Delta} \ar@{}[dr]|{(\ast)} & \Dec\bG \ar[d]^{\eps\bG} \\
\bH \ar@{=}[r] & \bH \ar[r]_{F} & \bG
}
\end{aligned}
\end{equation}
Finally, applying the functor $\pi_0$ to the upper rectangle, we get the factorization of $F$ into a final functor $J=(j_0,j_1)$ followed by a discrete fibration $K=(k_0,k_1)$:
\begin{equation} \label{diag:cf}
\begin{aligned}
\xymatrix{
H_1 \ar@<.7ex>[d]^{c} \ar@<-.7ex>[d]_{d} \ar[r]^{j_1} & T_1 \ar@<.7ex>[d]^{c} \ar@<-.7ex>[d]_{d} \ar[r]^{k_1} & G_1 \ar@<.7ex>[d]^{c} \ar@<-.7ex>[d]_{d} \\
H_0 \ar[r]_{j_0} & T_0 \ar[r]_{k_0} & G_0
}
\end{aligned}
\end{equation}


\section{Final functors}

The idea behind the characterization of final functors lies in the following lemma.

\begin{Lemma} \label{lemma:firstchar}
A functor $F\colon\bH\to\bG$ between internal groupoids in \cC\ is final if and only if its pullback $\overline{F}\colon\bE\to\Dec\bG$ along $\eps\bG$ is inverted by $\pi_0$.
\end{Lemma}

\begin{proof}
First of all, notice that $\pi_0(\overline{F})$ is nothing but the arrow $k_0$ in diagram (\ref{diag:cf}), so saying that $\overline{F}$ is inverted by $\pi_0$ means that the arrow $k_0$ is an isomorphism.

Now, if $F$ is final, $K$ is an isomorphism, and so is $k_0$. Conversely, since $K$ is a dicrete fibration, the right hand side commutative squares in (\ref{diag:cf}) are pullbacks, hence if $k_0$ is an isomorphism, so is $k_1$, and $F$ is final.
\end{proof}

We are now ready for an internal version of Proposition \ref{prop:final=fes}.

\begin{Theorem} \label{thm:main}
A functor between internal groupoids in a Barr-exact category is final if and only if it is full and essentially surjective.
\end{Theorem}

\begin{proof}
Let us consider the vertical expansion of the pullback $(\ast)$ in diagram (\ref{diag:Dec2}) and take the (bijective on objects, fully faithful) factorizations of the internal functors $F$ and $\overline{F}$:
\[
\xymatrix{
E_1 \ar[dr] \ar[rr]^{\phi_{\overline{F}}} \ar@<.7ex>[dd]^{c} \ar@<-.7ex>[dd]_{d} & & \overline{P} \ar[rr] \ar[dr]^{v} \ar@<.7ex>[dd] \ar@<-.7ex>[dd] & & R_c \ar[dr]^{\overline{d}} \ar@<.7ex>[dd]^(.35){r_2} \ar@<-.7ex>[dd]_(.35){r_1} \\
& H_1 \ar@<.7ex>[dd]^(.35){c} \ar@<-.7ex>[dd]_(.35){d} \ar[rr]^(.3){\phi_{F}} & & P \ar[rr] \ar@<.7ex>[dd] \ar@<-.7ex>[dd] & & G_1 \ar@<.7ex>[dd]^{c} \ar@<-.7ex>[dd]_{d} \\
E_0 \ar@{=}[rr] \ar[dr] & & E_0 \ar[rr]_(.75){\overline{f_0}} \ar[dr] & & G_1 \ar[dr]_{d} \\
& H_0 \ar@{=}[rr] & & H_0 \ar[rr]_{f_0} & & G_0
}
\]
Thanks to Lemma \ref{lemma:firstchar}, we only have to prove that $\overline{F}$ is inverted by $\pi_0$ if and only if $F$ is full and essentially surjective.

Suppose $\pi_0(\overline{F})$ is an isomorphism. Then, by Corollary \ref{cor:full->mono}, since $\Dec\bG$ is an equivalence relation, $\overline{F}$ is full, i.e.\ $\phi_{\overline{F}}$ is a regular epimorphism. But $v$ is a split epimorphism, being a pullback of $\overline{d}$, so $\phi_{F}$ is a regular epimorphism and $F$ is full.  Moreover, $\pi_0(\eps\bG)$ is a regular epimorphism since $\eps\bG$ is, so $\pi_0(F)\cdot\pi_0(U)=\pi_0(\eps\bG)\cdot\pi_0(\overline{F})$ is a regular epimorphism. This implies that $\pi_0(F)$ is a regular epimorphism, hence, by Proposition \ref{prop:es}, $F$ is essentially surjective.

Conversely, if $F$ is full, then $\overline{F}$ is full by Proposition \ref{prop:pbff}, hence $\pi_0(\overline{f})$ is a monomorphism by Corollary \ref{cor:full->mono}. If in addition $F$ is essentially surjective, $c\overline{f_0}=\pi_0(\overline{F})q_{\bE}$ is a regular epimorphism. So $\pi_0(\overline{F})$ is a regular epimorphism, hence an isomorphism.
\end{proof}

\begin{Corollary} \label{cor:thirdchar}
A functor $F\colon\bH\to\bG$ between internal groupoids in \cC\ is final if and only if
\begin{enumerate}
 \item[(i)] $\pi_0(F)$ is an isomorphism;
 \item[(ii)] the arrow $\psi_F$ in the following commutative diagram is a regular epimorphism:
\begin{equation} \label{diag:psi}
\begin{aligned}
\xymatrix@=2.5ex{
H_1 \ar[ddd]_{\sigma_H} \ar[rrr]^{f_1} \ar[dr]^{\psi_F} & & & G_1 \ar[ddd]^{\sigma_G} \\
& \Sigma H_1 \times_{\Sigma G_1} G_1 \ar[ddl]^\sigma \ar[urr] \\
\\
\Sigma H_1 \ar[rrr]_{\Sigma f_1} & & & \Sigma G_1
}
\end{aligned}
\end{equation}
\end{enumerate}
\end{Corollary}

\begin{proof}
By Proposition \ref{prop:ff<->mono}, $\pi_0(F)$ is a monomorphism if and only if $\Sigma F$ is fully faithful. In this case, the arrow $\psi_F$ coincides with the arrow $\phi_F$ defined in Section \ref{sec:functclasses}.

Suppose $F$ is final, then by Theorem \ref{thm:main} it is full and essentially surjective, hence, by Proposition \ref{prop:es} and Corollary \ref{cor:full->mono}, $\pi_0(F)$ is an isomorphism and $\phi_F=\psi_F$ is a regular epimorphism.

Conversely, if $(i)$ and $(ii)$ hold, then $\phi_F=\psi_F$ is a regular epimorphism, so $F$ is full, and moreover $\pi_0(F)$ is a regular epimorphism, i.e.\ $F$ is essentially surjective.
\end{proof}

\end{document}